\documentclass[12pt]{amsart}
\usepackage{amsfonts}
\usepackage{amssymb}
\usepackage{mathtools}
\usepackage[all]{xy}
\usepackage{amsmath, amssymb, amsfonts, amstext, amsthm}
\usepackage[mathscr]{euscript}
\usepackage{enumerate,color}
\usepackage{verbatim}
\usepackage{hyperref}
\usepackage{braket}
\usepackage{thmtools, thm-restate}
\theoremstyle{plain}
\def \ra{\rightarrow}
\def \mfk{\mathfrak}
\def \mcO{\mathcal{O}}
\def \mcOP{\mcO_{\mbb{P}^2}}

\def \mcalp3{\mcO_{\mbb{P}^3}}
\def \pd{\pi_D}
\def \mcPn{\mcO_{\mbb{P}^n}}
\def \mbb{\mathbb}

\newcommand{\hra}{\hookrightarrow}
\newcommand{\xra}{\xrightarrow}
\newcommand{\tildeB}{\widetilde{B}}

\newcommand{\tildeX}{\widetilde{X}}
\newcommand{\mbbP}{\mathbb{P}}
\newcommand{\mcalP}{\mathcal{O}_{\mathbb{P}^n}}
\newcommand{\mcE}{\mathcal{E}}
\newcommand{\mcS}{\mathcal{S}}
\newcommand{\mcOL}{\mcO_{\mbb{P}^1}}
\usepackage{hyperref}

\begin{document}
\begin{abstract}
In this article, we prove that any smooth projective variety $X$ which is a  double cover of the
  projective space $\mbb{P}^n$ ($n\geq 2$) admits an Ulrich bundle. When $n=2$, we show that on any such $X$, there is an Ulrich bundle of rank two. 
\end{abstract}

\newtheorem{thm}{Theorem}[section]
\newtheorem{lemma}[thm]{Lemma}

\newtheorem*{thmfour}{Theorem 4.4}

\title{Ulrich bundles on double covers of projective spaces}
\author[N.~Mohan Kumar]{N.~Mohan Kumar}
\address{Department of Mathematics, Washington University in St. Louis, St. Louis, Missouri, 63130}
\email{kumar@wustl.edu}

\author[P.~Narayanan]{Poornapushkala Narayanan}
\address{International Centre for Theoretical Sciences, Bangalore - 560089, and Madras School of Economics, Chennai - 600025}
\email{poorna@mse.ac.in}

\author[A.~J.~Parameswaran]{A.~J.~Parameswaran}
\address{School of Mathematics, Tata Institute of Fundamental Research, Mumbai - 400005.}
\email{param@math.tifr.res.in}

\thanks{Mathematics Classification numbers: 14E20, 14J60, 14J70, 14H30, 14H50}

\maketitle
\section{Introduction}
Arithmetically Cohen-Macaulay (ACM) vector bundles have been at the center of the study of vector bundles over projective varieties. These
are defined to be vector bundles $E$ on projective varieties $X$ satisfying
$$H^i(X,E(l))=0\quad\text{for all }l\in\mbb{Z}\text{ and }0<i<\text{dim}\,X\,.$$
It was proved by Horrocks \cite{Hor} that ACM vector bundles on
projective spaces $\mbb{P}^n$
are precisely those which split as a direct sum of line 
bundles. Following this, there has been  a lot of interest in understanding ACM bundles over all varieties.
Among ACM bundles $E$, Ulrich bundles are those whose associated module
$\oplus_{t} H^0(X,E(t))$ has the maximum number of generators. 

In his 1984 paper \cite{Ul}, Ulrich asked whether every smooth projective variety carries an Ulrich bundle.
Eisenbud and Schreyer \cite{ES} drew attention to this question and also asked for the minimal rank of such a bundle whenever it exists. 
The existence of Ulrich bundles has been established in the case of several varieties, we refer to \cite{AB} and \cite{Co} and references therein for a complete list of such varieties. 

In this note, we consider complex smooth projective varieties $X$ which admit a degree two morphism $\pi:X\ra\mbb{P}^n$ (for $n\geq 2$) to the projective space.  We discuss the existence of Ulrich bundles over $X$  with respect to the morphism $\pi$, or equivalently with respect to the ample and globally generated line bundle $\pi^*\mcPn(1)$. Observe that a vector bundle $E$ on $X$ is Ulrich with respect to $\pi$, if $\pi_*E$ is the trivial vector bundle on $\mbb{P}^n$. When $n=2$, i.e. if $X$ is a double cover of $\mbb{P}^2$, we prove the following result regarding the existence of Ulrich bundles on $X$ in $\mathcal{x}\,$\ref{Ulrich bundles over surfaces}.
\begin{thm}\label{UB on the plane}
 A smooth double cover $X\xra{\pi} \mbb{P}^2$ of the projective plane branched along a 
 smooth curve $B$ admits a rank two Ulrich bundle.
\end{thm}
In \cite{AJP}, the authors investigated the existence of Ulrich line bundles over double covers of $\mbb{P}^2$ and showed that a smooth double cover of $\mbb{P}^2$ branched
 over a curve of degree $2d$ when $d=1,2$ admits Ulrich \emph{line bundles}. When $d\geq 3$, it was shown
 that a generic double cover does not admit an Ulrich line bundle, but there are special classes of double
 covers which do indeed admit Ulrich line bundles. 
In \cite{RA}, Sebastian and Tripathi prove that general plane double covers admit rank two Ulrich bundles using different techniques. We remark that Theorem \ref{UB on the plane} completely answers the question of existence and minimal rank of Ulrich bundles over all smooth double covers of $\mbb{P}^2$.

The proof of the above theorem involves the analysis of double covers of smooth curves $p:C\ra D$ and identifying the conditions on the curve $C$ for it to  admit a line bundle whose direct image is trivial, cf. Theorem \ref{pencil}. This condition is indeed met when $D$ is a smooth plane curve of degree $d$, and this will be used in the proof of Theorem \ref{UB on the plane}.

We remark that, previously, we had attempted to use \cite[Theorem A]{KO} which also studies line bundles on curves $C$ where $p:C\ra D$ is a double cover of smooth curves. However, we realised that the proof of Theorem A in \cite{KO} is not quite correct and we provide a counterexample to their result in the Appendix, cf. $\S$ \ref{appendix}.

In case of double covers of $\mbb{P}^n$ when $n\geq 3$, we prove the following theorem.
\begin{thm}\label{Ulrich bundle on higher Pn}
  Let $X\xra{\pi} \mbb{P}^n$ be a smooth double cover branched over a smooth hypersurface $B$ of
 degree $2d$. Then $X$ admits an Ulrich bundle.
\end{thm}
We give two different proofs of the above theorem in $\S$ \ref{higher double covers}. The first approach involves embedding $\mbb{P}^n$ in a higher dimension projective space $\mbb{P}^N$ by the $d$-uple embedding and understanding the existence of Ulrich sheaves on double covers of $\mbb{P}^N$ branched along quadrics. In this regard, we use the well known and seminal result of Herzog, Ulrich and Backelin \cite{HUB} regarding the existence of Ulrich sheaves on complete intersection varieties $Z$. 
We reproduce their proof in the special case when $Z$ is quadric hypersurface for the sake of completeness, cf. Theorem \ref{quadric hypersurface}.

In the second approach, we consider appropriate complete intersection subvarieties $Y$ inside $\mbb{P}^n$ and their double covers $\widetilde{Y}$ inside $X$. We prove that the existence of vector bundles on $\widetilde{Y}$ whose direct image in $Y$ is trivial give rise to Ulrich bundles on $X$, cf. Theorem \ref{induction theorem}. We then show that such vector bundles indeed exist on $\widetilde{Y}$, cf. Lemma \ref{vb direct image}.

We make some observations about the ranks of the Ulrich bundles obtained on the double covers $X$ of $\mbb{P}^n$ when $n\geq 3$ in $\mathcal{x}\,$ \ref{rank section}. Theorem \ref{max rank UB} gives an upper bound for the minimal rank of Ulrich bundles on double covers $X$ of $\mbb{P}^n$ when $n\geq 3$. However, computing the minimal rank is still an open question which we are investigating.

\section{Preliminaries}\label{prelim}
\subsection{Ulrich bundles}\label{prelim-ub}
Here we briefly introduce Ulrich bundles, cf. \cite{AB} and \cite{Co} for thorough surveys on Ulrich bundles.
\newtheorem{defn}[thm]{Definition}
\begin{defn}\label{ulrich}
Let $X$ be a smooth projective variety over $\mbb{C}$ together with a finite morphism $\pi:X\ra\mbb{P}^{\,\emph{dim}\,X}$ of degree $m$. A vector bundle $E$ of rank $r$ on $X$ is said to be Ulrich with respect to the morphism $\pi$ if $\pi_*E=\mcO_{\mbb{P}^{\,\emph{dim}\,X}}^{\oplus\,mr}$. 

If we denote $\pi^*\mcO_{\mbb{P}^{\,\emph{dim}\,X}}(1)=\mcO_X(1)$, then $E$ is Ulrich on $X$ if and only if $H^i(X,E(-p))=0$ for $1\leq p\leq\emph{dim}\,X$ and for all $i$.
\end{defn}
\subsection{Secant varieties}\label{Secant varieties prelim}
Given a subvariety $X$ of $\mbb{P}^n$, we can define the higher order secant varieties of $X$ as follows.
\begin{defn}\label{join}
Let $X$ be a variety in $\mbb{P}^n$, then the Secant variety of order $(r-1)$ of $X$ is the closure of the union of all
linear subspaces of $\mbb{P}^n$ spanned by $r$ points in $X$. More precisely,
$$\emph{Sec}_{r-1}X=\overline{\bigcup_{v_1,v_2,\ldots,v_r\in {X}}\langle{v_1,v_2,\ldots,v_r}\rangle}\subset \mbb{P}^n$$
where for $S\subset \mbb{P}^n$, the notation $\langle S \rangle$ denotes
the smallest linear subspace of $\mbb{P}^n$ containing $S$.
\end{defn}

We are interested in the higher order secant varieties of a particular projective variety $\mbb{X}$ which is the set of degree $2d$ hypersurfaces in $\mbb{P}^n$ which
are reducible into two degree $d$ hypersurfaces, i.e.   
\begin{equation}\label{reducible 2d hypersurfaces}
\mbb{X}=\{[F]\in \mbb{P}H^0(\mbb{P}^n,\mcPn(2d))\,|\,F=F_1F_2 \text{ and } F_i\in H^0(\mbb{P}^n,\mcPn(d))\}\,.
\end{equation}
It is known that $\mbb{X}\subset \mbb{P}H^0(\mbb{P}^n,\mcPn(2d))$ is a subvariety with dimension 
$2{d+n\choose n}-2$, cf. \cite{Mam}, \cite{CCG}. 

Note that if $[F]\in \text{Sec}_{r-1}\mbb{X}$, then $F$ can be
written as
$$F=\Sigma_{i=1}^r F_i G_i\ \text{where }F_i,G_i\in H^0(\mbb{P}^n,\mcPn(d))\,.$$
It is clear that we have a chain of inclusions
$$\mbb{X}\subset \text{Sec}_{1}\mbb{X}\subset\cdots\subset 
\text{Sec}_{r-1}\mbb{X}\subset \text{Sec}_{r}\mbb{X}\subset\cdots\subset\mbb{P}V\,.$$
We will see in Lemma \ref{filtration covers} that there is an integer $m$ such that $\text{Sec}_{m}\mbb{X}$ covers $\mbb{P}V$.
\section{Polynomial expression of a plane curve}\label{SSP section}
Let $F\in \mbb{C}[x,y,z]$ be a smooth homogeneous polynomial of degree $2d$. Let $[F]\in\mbb{P}H^0(\mbb{P}^2,\mcOP(2d))$ denote the corresponding hypersurface of $\mbb{P}^2$. In \cite{CCG}, the authors show that $[F]\in\text{Sec}\,_1\mbb{X}$ where $\mbb{X}$ is described in equation \eqref{reducible 2d hypersurfaces}, i.e. $F$ can be written in the form $$F=F_1G_1+F_2G_2\,,$$
where $F_1,G_1,F_2,G_2\in\mbb{C}[x,y,z]$ are homogeneous polynomials of degree $d$. We broadly explain their ideas in our context. The following is the key lemma.
\begin{lemma}\label{Carlini}\cite[Lemma 4.1, Lemma 4.3]{CCG}
Let $S=\mbb{C}[x,y,z]$. The following statements are equivalent. 
\begin{enumerate}
 \item The generic degree $2d$ curve $(F=0)$ in $\mbb{P}^2$ can be written in the form $F=F_1G_1+F_2G_2$ where $F_1,F_2,G_1,G_2\in S$ are of degree $d$.
 \item The Secant variety $\emph{Sec}_1\mbb{X}$ covers the space of degree $2d$ forms in 3 variables i.e.
 $$\emph{Sec}_1\mbb{X}=\mbb{P}(S_{2d}).$$ 
 \item Let $F_1,F_2,G_1,G_2\in S$ be generic forms of degree $d$, then 
 $$H\Big(\frac{S}{(F_1,F_2,G_2,G_1)},2d\Big)=0$$
 where $H(\_,2d)$ denotes the Hilbert function in degree $2d$ of the ring.
\end{enumerate}
\end{lemma}
In order to prove that a generic $F\in S_{2d}$ has an expression of the form 
$$F=F_1G_1+F_2G_2 $$
where $F_1,F_2,G_1,G_2\in S_d$, the authors in \cite{CCG} use condition (3) of the above lemma i.e. they show that the corresponding Hilbert function is zero. In particular, we refer to the proof of \cite[Theorem 5.1]{CCG}. Here they consider the ring 
\begin{equation}\label{Artin ring}
A=\frac{\mbb{C}[x,y,z]}{(F_1,F_2,G_2)}
\end{equation}
where $F_1,F_2,G_2$ are generic forms of degree $d$. They then show that $H\Big(\frac{A}{\overline{G_1}},2d\Big)=0\,,$
where $\overline{G_1}$ is the image of of a generic degree $d$ form $G_1\in \mbb{C}[x,y,z]$ in the ring $A$. Showing $H\Big(\frac{A}{\overline{G_1}},2d\Big)=0$ is equivalent to showing that the multiplication by $\overline{G}_1$ map:
\begin{equation}\label{multiplication map}
m:A_d\ra A_{2d} 
\end{equation}
is surjective. 

We now prove condition (3) of Lemma \ref{Carlini} by showing that the morphism $m$ above is surjective for generic choices of $F_1,F_2, G_1, G_2$, by considering certain extensions of vector bundles in this upcoming subsection.
\subsection{Extensions}
Let $C$ be a smooth projective curve
of genus $g$ with canonical bundle $K_C$. Let $L$ be a line bundle with degree greater than $2g - 2$. We wish to consider
extensions,
$$0\ra \mcO_C\ra E\ra L\ra 0\,.$$
These are parametrized by $\text{Ext}^1(L,\mcO_C)=H^1(L^{-1})$. We denote this vector space (or affine space when applicable) by $V$. 

Let $V_s\subset V$ be the subset consisting of elements $\eta \in V$ such that the corresponding extension $E_{\eta}$ has $H^1(E_{\eta})\neq 0$. Dual to our extension, we have
$$0\ra L^{-1}\otimes K_C\ra E^{\vee}\otimes K_C\ra K_C\ra 0\,.$$
Let $M=L^{-1}\otimes K_C$ and $F=E^{\vee}\otimes K_C$, then $\eta\in V_s\subset V=\text{Ext}^1(K_C,M)$ is just the subset where $H^0(F_{\eta})\neq 0$ by Serre duality. Note that $\text{deg}\,M=-e$ with $e>0$ by our assumption on the degree of $L$.

\begin{thm}\label{3g-3}
 Let $M$ be a line bundle of degree $-e$, with $e > 0$ and notation as above. Then $V_s\subset V$ is a closed subset of dimension at most $3g - 3$,
where for $g = 0$, this is to be understood as $V_s=\emptyset$.
\end{thm}
\begin{proof}
 If $g = 0$, the result is obvious since both $M$ and $K_C$ have negative degrees and thus every extension has zero global sections. So,
we assume $g > 0$ for the rest of the proof.

Let $\mbb{P}=\mbb{P}(H^0(K_C))=\mbb{P}^{g-1}$ , the set of non-zero sections of $K_C$ .
Consider $Z \subset \mbb{P} \times V = X$ consisting of $(\alpha,\eta)$ where $\alpha$ is a non-zero section of $K_C$ and $\alpha$ lifts to $F_{\eta}$ for the extension 
$$0\ra M \ra F_{\eta} \ra K_C \ra 0\,.$$
Thus, the image of $Z$ in $V$ under projection is contained in $V_s$.
Conversely, if $\eta\in V_s$, we have a non-zero section of $F_{\eta}$ , which by
composition gives a non-zero section of $K_C$, since $\text{deg}\,M<0$. This
says that $V_s$ is precisely the image of $Z$ under the projection to $V$.

On $\mathbb{P}$, we have the universal map
$\mcO_{\mathbb{P}}(-1)\ra H^0(K_C)\otimes \mcO_{\mathbb{P}}$ which we pull back to $X$. Similarly, on $V$ we have the universal map $H^0( K_C )\otimes\mcO_V\ra
H^1 ( M )\otimes\mcO_V$ , cupping with $\eta\in V = H^1 ( M \otimes K_C^{-1} )$ which also can be pulled back to $X$. Then we get a composition $\mcO_X(-1)\ra H^1( M )\otimes \mcO_X$ . Then $Z$
is just the points on $X$ where this map is zero and thus $Z$ is a closed
subset of X and since $V_s$ is the image under the proper map to $V$, we
see that $V_s$ is a closed subset of $V$.

We calculate the dimension of $Z$ by looking at the projection to $\mbb{P}$.
If $(\alpha,\eta)\in Z$ , we look at the pull back diagram,
\begin{displaymath}
 \xymatrix{0 \ar[r] & M \ar[r] & G \ar[d]\ar[r] & \mcO_C\ar[r]\ar[d]^{\alpha}\ar@^{.>}[dl] & 0\\
 0 \ar[r] & M\ar@^{=}[u]\ar[r] & F_{\eta} \ar[r] & K_C\ar[r] & 0}
\end{displaymath}
The section $\alpha$ lifts to $F_{\eta}$ says the sequence $0 \ra M \ra G \ra\mcO_C \ra 0$
splits and conversely. So the fibre of $\alpha\in \mbb{P}$ are the extensions in the kernel of the induced map, $V = \text{Ext}^1 ( K_C , M ) \ra
\text{Ext}^1 (\mcO_C , M )$ .

We have an exact sequence 
$$ 0\ra \mcO_C \ra K_C \ra T \ra 0 $$ where $T$
is a sheaf of length $2g - 2$. 
Taking $\text{Hom}(\underline{\hspace{0.3 cm}}\,,M)$ , we get,
$$\ra 0 = H^0 ( M ) \ra \text{Ext}^1 ( T, M ) \ra \text{Ext}^1 ( K_C , M ) = V \ra \text{Ext}^1 (\mcO_C , M )\ra \,.$$
The kernel of the map $V \ra \text{Ext}^1 (\mcO_C,M)$ , which is precisely $\text{Ext}^1( T, M )$ is
a vector space of dimension $2g - 2$. Thus, $Z \ra \mathbb{P}$ is an affine bundle
of fiber dimension $2g - 2$ and $\text{dim}\, Z = 3g- 3$. So, $\text{dim}\, V_s \leq
3g - 3$.
\end{proof}
\newtheorem{coro}[thm]{Corollary}
\newtheorem{rmk}[thm]{Remark}
\begin{coro}
For a general element $\eta\in V$, the extension 
$$0 \ra \mcO_C \ra E_{\eta} \ra L \ra 0$$
where $\emph{deg}\,L > 2g - 2$ has $H^1 ( E_{\eta} ) = 0$.
\end{coro}
\begin{proof}
 We only need to show that $V_s\neq V$. But 
 $$\text{dim}\, V = \text{deg}\, L + g -1>3g-3\,.$$
\end{proof}
We next have the following definition.
\begin{defn}
 For a vector bundle $E$, we say a point $P\in C$ is a base point,
if the natural map $H^0 ( E ) \otimes \mcO_C \ra E$ is not surjective at $P$, That is, the
cokernel has $P$ in its support.
\end{defn}
We denote by $V_{bp}$ , the set of $\eta\in V=H^1 ( L^{-1} ) = \text{Ext}^1 ( L, \mcO_C )$ such that the corresponding extension $E_{\eta}$ has a base point.
\begin{thm}\label{3g-1}
Assume that $\emph{deg}\,L = r > 2g$. Then $V_{bp}$ is contained in a subvariety of dimension at most $3g - 1$.
\end{thm}
\begin{proof}
 We restrict our attention to the open set $U = V -V_s$ , and
$V_{bp} \cap U$, since $\text{dim} V_s \leq 3g - 3$ by the previous theorem.
Let $X = C \times U$ and $p : X \ra C$, $q : X \ra U$ the two projections. On
$X$ we have the universal extension,
$$0\ra \mcO_X\ra \mcE \ra p^*L \ra 0\,.$$
Since for any $\eta\in U$, one has $H^1( E_{\eta} ) = 0$, we see that $q_*\mcE$ is a vector bundle ($q$ is proper and flat) and so we get a map $q^*q_* \mcE \ra \mcE$ . Let $Z$
be the support of the cokernel of this map, which is a closed subset
of $X$. It is immediate that $q ( Z ) = V_{bp} \cap U$ and thus, it is a closed
subset of $U$.

We calculate the dimension of $Z$ by using the projection $p$. For a point $P\in C$, the fibre in $Z$ over $p$ is the set of all $\eta\in U$ such that $P$ is a base point of $E_{\eta}$.


A pair $(P,\eta)\in Z$ if and only if $P$ is
a base point of $E_{\eta}$ . However, this says, the image of $H^0( E_{\eta} ) \ra H^0 ( L )$ factors through $H^0(L(- P))$.



We look at the following pullback diagram.
\begin{displaymath}
 \xymatrix{
 0\ar[r] & \mcO_C\ar[r]\ar@{=}[d] & F\ar[r]\ar[d] & L(-P)\ar[r]\ar[d] & 0\\
 0\ar[r] & \mcO_C\ar[r] & E_{\eta}\ar[r] & L\ar[r] & 0}
\end{displaymath}
Note that $H^0(E_{\eta})\ra H^0(L_P)$ is the zero map where $L_P$ comes from the exact sequence
$$0\ra L(-P)\ra L\ra L_P\ra 0\,.$$
This gives us $H^0(F)=H^0(E_{\eta})$.

If $e$ is the dimension of the image of $H^0(E_{\eta}) \ra H^0(L)$, one has
$e + g = h^0(L)$, since $H^1( E_{\eta} ) = 0$. But, by assumption, $e$ is also the dimension of the image of $H^0(F) \ra H^0(L(-P))$ and since $h^0(L(-P))=h^0(L) - 1$, we see that $H^1 (F)  \neq 0$.

Let $W = H^1 ( L^{-1}(P)) = \text{Ext}^1(L(-P),\mcO_C)$, and $W_s\subset W$ be defined as earlier. Since $H^1( F ) \neq 0$, the class corresponding to that extension in $W$ must infact be in $W_s$ . 

Taking $\text{Hom}(\underline{\hspace{0.3cm}}\,, \mcO_C )$ of the exact sequence,
$$0 \ra L(-P) \ra L \ra L_P \ra 0\,$$
we get,
$$0 \ra \text{Ext}^1(L_P,\mcO_C) \ra V \ra W \ra 0\,.$$
Consider $(P,\eta)\in Z$. Then the image of $\eta$ under the map $V\ra W$ lands in $W_s$ as seen above. Thereby, such an $\eta$ is in the inverse image of $W_s$. 

Since $\text{deg}(L(-P))=\text{deg}\,L-1>2g-2$, by Theorem \ref{3g-3}, we get that $\text{dim}\, W_s$ is atmost $3g-3$. Also 
$\text{Ext}^1(L_P,\mcO_C)$ is a one dimensional vector space. So the inverse image
of $W_s$ in $V$ has dimension at most $3g - 2$. This is true for any $P \in C$
and thus $\text{dim}\,Z \leq 3g -1$, which implies $\text{dim}\,V_{bp} \cap U \leq 3g - 1$.
\end{proof}
\begin{coro}
 If $\emph{deg}\,L > 2g$, a general extension $$ 0 \ra \mcO_C \ra E \ra L \ra 0$$
has $H^1(E)=0$ and $E$ is globally generated.
\end{coro}
\begin{proof}
 As before, letting $V=H^1(L^{-1})$ , $\text{dim}\,V = \text{deg}\, L + g - 1$ by
Riemann-Roch and $\text{deg}\, L + g - 1 > 3g -1$. On the other hand, $\text{dim}\, V_{bp}\leq 3g-1$ and $\text{dim}\,V_s\leq 3g-3$.
\end{proof}
\begin{rmk}
 Both bounds obtained for the line bundles in this section are the best possible in general.
 \begin{itemize}
  \item For example, if we allow $\emph{deg}\,L = 2g - 2$ in Theorem \ref{3g-3}, then we may take $L = K_C$ and since for any extension $E_{\eta}$, we have a surjection $H^1(E_{\eta}) \ra H^1(K_C) \neq 0$,
we see that all extensions are special, i.e. $V_s=V$.
\item Suppose we allow $\emph{deg}\,L = 2g$ in Theorem \ref{3g-1}, then a
general extension $E_{\eta}$ is non-special i.e. $H^1(E_{\eta})= 0$. From the sequence
$$ 0\ra\mcO_C\ra E_{\eta}\ra L\ra 0$$
we get the surjective map $H^0(L)\ra H^1(\mcO_C)$. Thus the dimension of the
image of $H^0(E_{\eta}) \ra H^1(L)$ is equal to $h^0(L)-g = 1$. If $g > 0$, one
section can not generate $L$ and thus $E_{\eta}$ has base points.
 \end{itemize}
\end{rmk}

Finally, we derive the corollaries we want.
\begin{coro}\label{deg d curve}
 Let $C \subset \mbb{P}^2$ be a smooth curve of degree $d$. Then we have an
extension 
$$0 \ra \mcO_C \ra E \ra \mcO_C(d) \ra 0$$ such that, $H^1(E) = 0$ and $E$ is
globally generated.
\end{coro}
\begin{proof}
 By Theorem \ref{3g-3} and Theorem \ref{3g-1}, we just need to compute $\text{deg}\,\mcO_C(d)$:
 $$\text{deg}\,\mcO_C(d)=d^2>2\frac{(d-1)(d-2)}{2}=2g(c)\,.$$
\end{proof}
\begin{coro}
 For any $d > 0$, there exist four homogeneous polynomials
in three variables of degree $d$ so that the ideal generated by them contains
all degree $2d$ homogeneous polynomials. (This is an open condition, so the
conclusion is true for a general set of four degree d polynomials.)

In particular, for a general set of four polynomials $F_1,F_2,G_1,G_2$ the multiplication map 
$$A_d=\frac{\mbb{C}[x,y,z]}{(F_1,F_2,G_2)}_d\xra{\overline{G_1}} \frac{\mbb{C}[x,y,z]}{(F_1,F_2,G_2)}_{2d}=A_{2d}$$
is surjective.
\end{coro}
\begin{proof}
 Let $C \subset \mbb{P}^2$ be a smooth curve of degree $d$ and let 
 $$0 \ra \mcO_C \ra E \ra \mcO_C(d) \ra 0$$
 be an extension such that $H^1(E) = 0$ and $E$ globally generated. Then we can pick three general sections which generate
$E$ to get a sequence, 
$$0 \ra \mcO_C(-d) \ra \mcO_C^3 \ra E \ra 0\,.$$
Dualize and twist by $d$ to get a sequence,
$$0 \ra E^{\vee}(d) \ra \mcO_C(d)^3 \ra \mcO _C(2d) \ra 0\,.$$
But $E^{\vee}(d) \simeq E$ and thus $H^1(E^{\vee}(d)) = 0$. This implies that the map $H^0(\mcO_C(d))^3 \ra H^0(\mcO_C(2d))$ is onto.

In particular, if the curve $C$ is given by the equation $(G_1=0)$, we get that the following map is onto
$$m_{F_1,F_2,G_2}:\left(\frac{\mbb{C}[x,y,z]}{(G_1)}\right)_d^3\xra{F_1,F_2,G_2} \left(\frac{\mbb{C}[x,y,z]}{(G_1)}\right)_{2d}\,.$$
This means that, given any degree $2d$ homogeneous polynomial $Q(x,y,z)$, there exist degree $d$ homogeneous polynomials $f,g,h\in \mbb{C}[x,y,z]$ such that
$$Q-fF_1-gF_2-hG_2\in (G_1)\,.$$
It is now easy to see the equivalence of  $m_{F_1,F_2,G_2}$ being surjective and the following map $m$ being surjective
$$m: \frac{\mbb{C}[x,y,z]}{(F_1,F_2,G_2)}_d\xra{\times\, \overline{G_1}} \frac{\mbb{C}[x,y,z]}{(F_1,F_2,G_2)}_{2d}\,.$$
\end{proof}
The above theorem shows that the morphism $m$ as defined in \eqref{multiplication map} is indeed surjective and thus $H\Big(\frac{A}{\overline{G_1}},2d\Big)=0$.

Thus the generic degree $2d$ curve $(F=0)$ in $\mbb{P}^2$ can be written in the form $F=F_1G_1+F_2G_2$ where $F_1,F_2,G_1,G_2\in S$ are of degree $d$. 

\begin{rmk}
We remark that \emph{every} degree $2d$ form $F\in S$ can be written in the form $F=F_1G_1+F_2G_2$ where $F_1,F_2,G_1,G_2\in S$ are of degree $d$. The proof of \cite[Theorem 5.1]{CCG} shows that $\emph{Sec}\,\mbb{X}=\mbb{P}(S_{2d})$ where $\mbb{X}$ is as defined in \eqref{reducible 2d hypersurfaces}. \cite[Remark 5.4]{CCG} explains that points of the variety of secant lines lie on a true secant line or on a tangent line to $\mbb{X}$. However, if $q$ is a point on the tangent line to $\mbb{X}$ at a point $p=[FG]\in \mbb{X}$, then $q$ can be written in the form 
$[\alpha FG'+\beta F'G]$ for some forms $F',G'$ of degree $d$ and scalars $\alpha,\beta$, and thus $q$ lies on the secant line joining $[F'G]$ and $[FG']$.  
\end{rmk}
\begin{lemma}
 Let $F\in S$ be smooth of degree $2d$. Then $F$ can be written in the form $F=F_1'G_1'+F_2'G_2'$ where $F_i'$ and $G_i'$ have degree $d$,  $F_1'$ is smooth and $F_2'$ and $G_2'$ intersect transversely at $d^2$ distinct points. 
\end{lemma}
\begin{proof}
By the previous remark, any $F\in S_{2d}$ can be written in the form 
$$F=F_1G_2+F_2G_2\text{ where } F_i,G_i\in S_d\,.$$
 Consider the linear system $\mbb{P}V$ associated to the vector space $V$ generated by $\{F_1,F_2\}$. Any element of $V$ has the form $aF_1+bF_2$ where $a,b\in \mbb{C}$. Hence, a general element of the linear system $\mbb{P}V$ looks like $F_1+\alpha F_2$ where $\alpha\in \mbb{C}$. 
 
 By Bertini's theorem, a general element of $\mbb{P}V$ is non-singular away from the basepoints of the linear system, i.e. the points $\mathfrak{b}=\{P\in\mbb{P}^2\,|\,F_1(P)=0=F_2(P)\}\,.$ 
 
Let $P\in \mathfrak{b}$. Then note that at least one of $F_1$ or $F_2$ is non-singular at $P$. Indeed, if both are singular at $P$ , then all the first partials $F_x,F_y,F_z$ of $F$ vanish at $P$. 
That is, $F$ itself is singular at $P$ which is a contradiction. 
Thereby, there is a dense open subset $U_1\subset\mbb{C}$ such that for $\alpha\in U_1$, $F_1+\alpha F_2$ is non-singular at all points $P\in\mfk{b}$. 
For $\alpha\in U_1$ we rewrite $F$ in the form
$$F=F_1G_1+F_2G_2=(F_1+\alpha F_2)G_1+F_2(G_2-\alpha G_1)={F}_{\alpha} G_1+F_2{G}_{\alpha}\,.$$

We now need to ensure the transversal intersection of the second component of the summation above. 

Consider the linear system generated by $\{{F}_{\alpha},F_2\}$. A general element of the linear system is smooth away from the basepoints of the linear system. By a similar argument as above $F_2+\beta ({F}_{\alpha})$ is a non-singular element of the linear system whenever $\beta\in U_2$ where $U_2\subset\mbb{C}$ is a dense open subset. Now we can rewrite $F$ as
$$F={F}_{\alpha}G_1+F_2{G}_{\alpha}
={F}_{\alpha}(G_1-\beta{G}_{\alpha}) +(F_2+\beta {F}_{\alpha}){G}_{\alpha}={F}_{\alpha}{G}_{\beta}+F_{\beta}G_{\alpha}\,,$$
where $F_{\alpha}$ and $F_{\beta}$ are smooth. 

The linear system defined by $\{G_1,G_2\}$ give the following rational map
$$\phi:\mbb{P}^2\dashrightarrow \mbb{P}^1$$
which is a morphism away from the basepoints of the linear system. The fiber whenever it makes sense is defined by $\{G_{\alpha}=0\}=\{G_2-\alpha G_1=0\}$.
Consider the restriction of $\phi$ to smooth curves $(F_{\beta}=0)$ where $\beta\in U_2$.
\begin{displaymath}
 \xymatrix{\mbb{P}^2\ar@^{..>}[r]^{\phi} & \mbb{P}^1 \\ 
            (F_{\beta}=0)\ar@^{..>}[ur]\ar@{^{(}->}[u] & }
\end{displaymath}
The fibers of this restriction precisely give the intersection of $F_{\beta}$ with $G_{\alpha}$. The differential of this restriction is non-zero since $(F_{\beta}=0)$ is smooth and hence, separated. Therefore, the tangent directions of $F_{\beta}$ and $(G_{\alpha})$ are independent for a general $\alpha$ and a general $\beta$  implying that $F_{\beta}$ and $G_{\alpha}$ intersect transversally.

\end{proof}

\section{Ulrich bundles on double covers of the projective plane}\label{Ulrich bundles over surfaces}
In this section, we use induction on the dimension of the variety to prove the existence of Ulrich bundles on smooth double covers. We start with the following general theorem.
\begin{thm}\label{induction theorem}
Consider $\pi:X\ra Y$ a double covering of smooth projective varieties of dimension $n>1$ defined over $\mbb{C}$. Let $\pi_*\mcO_X\simeq \mcO_Y \oplus L^{-1}$ where $L\in \emph{Pic}\, Y$ and 
$\pi$ is branched over a smooth hypersurface $B\in |L^{\otimes 2}|$. 
Suppose that
\begin{enumerate}
 \item $H^1(Y,\mcO_Y)=0$, and
 \item for a smooth and irreducible divisor $D\in |L|$ which is such that $\pi|_D:X_D:=X\times_{Y} D\ra D$ is a double covering,  
there is a vector bundle $E$ of rank $r$ on $X_D$ whose direct image in $D$ is trivial. 
\end{enumerate}
Then there is a vector bundle $E'$ on $X$ of rank $2r$ whose direct image in $Y$ is trivial.
\end{thm}
\begin{proof}
We note that for a general divisor $D\in |L|$, the morphism $\pi_D:X_D\ra D$ is a double cover and
we have the following fibered diagram.
\begin{displaymath}
\xymatrix{X_D \ar@{^{(}->}[r]^i\ar[d]_{\pd} & X\ar[d]^{\pi} \\
 D \ar@{^{(}->}[r]_{j} & Y}
\end{displaymath}
By hypothesis, there is such a smooth and irreducible $D\in |L|$ and a vector bundle $E$ on $X_D$ such that $(\pi_D)_*E\simeq \mcO_D^{\oplus\, 2r}$. The finite morphism $\pi_D$ gives the surjection
 \begin{equation}\label{global generation}
  \mcO_{X_D}^{\oplus\, 2r}\simeq (\pi_D)^*(\pi_D)_*E\ra E\ra 0\,,
 \end{equation}
 which shows that $E$ is globally generated. Also, 
 $$h^0(X_D,E)=h^0(D,(\pd)_*E)=h^0(D,\mcO_D^{\oplus\, 2r})=2r\,.$$
Pushing forward the surjection \eqref{global generation} to $X$, consider the following composition
$$\mcO_X^{\oplus\, 2r}\twoheadrightarrow i_*\mcO_{X_D}^{\oplus\,2r}\ra i_*E\ra 0\,.$$
We get the following short exact sequence on $X$ where $G$ denotes the kernel of the above surjection:
$$0\ra G\ra \mcO_X^{\oplus\, 2r}\ra i_*E\ra 0\,.$$
Since $G$ is obtained by the process of elementary transformation of the trivial vector bundle on $X$ along a divisor $X_D\subset X$, the sheaf $G$ is locally free and has rank $2r$. 
We next push forward the above short exact sequence to $Y$:
$$0\ra \pi_*G\ra \pi_*\mcO_X^{\oplus 2r}\ra \pi_*i_*E\ra 0.$$
Since $\pi_*\mcO_X^{\oplus 2r}\simeq (\mcO_Y\oplus L^{-1} )^{\oplus\, 2r}$ and $\pi_*i_*E\simeq j_*\mcO_D^{\oplus 2r}$, we have the following commutative diagram.
\begin{displaymath}
 \xymatrix{
  & 0\ar[d] & 0\ar[d] & 0\ar[d]  & \\
 0\ar[r] & (L^{-1})^{\oplus\,2r}\ar[r]\ar@{^{(}->}[d] & \mcO_Y^{\oplus\,2r}\ar[r]\ar@{^{(}->}[d] & j_*\mcO_D^{\oplus\,2r}\ar[r]\ar[d]^{\simeq} & 0\\
0\ar[r] & \pi_*G\ar[r]\ar[d] & \mcO_Y^{\oplus\, 2r}\oplus (L^{-1})^{\oplus\, 2r}\ar[r]\ar[d] & j_*\mcO_D^{\oplus\,2r}\ar[d]\ar[r] & 0 \\
0 \ar[r] & (L^{-1})^{\oplus\,2r}\ar@{=}[r]\ar[d] & (L^{-1})^{\oplus\,2r}\ar[r]\ar[d] & 0 &\\
& 0 & 0 & & }
\end{displaymath}
Hence $\pi_*G$ is a vector bundle on $Y$ given by the extension:
$$0\ra (L^{-1})^{\oplus\,2r}\ra \pi_*G\ra (L^{-1})^{\oplus\,2r} \ra 0\,.$$
However, 
$\text{Ext}^1((L^{-1})^{\oplus\,2r},(L^{-1})^{\oplus\,2r})\simeq 
H^1(Y, \mcO_Y^{\oplus\,4r})=0\,$ by assumption. Hence, we get $\pi_*G\simeq (L^{-1})^{\oplus\,2r}\oplus (L^{-1})^{\oplus\,2r}$. Thereby, the required vector bundle is $G\otimes \pi^*L$, since by projection formula, 
$$\pi_*(G\otimes\pi^*L)\simeq \pi_*G\otimes L\simeq (L^{-1})^{\oplus\,4r}\otimes L\simeq \mcO_Y^{\oplus\, 4r}\,.$$
\end{proof}
The above theorem will be used to show that smooth double covers of $\mbb{P}^2$ carry rank two Ulrich bundles. To that end, we examine in the upcoming subsection, double covers $p:C\ra D$ of smooth curves and identify when the curve $C$ has line bundles $A$ whose direct image on $D$ is trivial.
\subsection{Line bundles on double covers of curves}

Let $p:C\ra D$ be a double cover of smooth curves and $p_*\mcO_C=\mcO_D\oplus L^{-1}$ for a line bundle $L$ on $D$ with $\text{degree}\,L=d$. Let the genus of $C$ and $D$ be $g$ and $h$ respectively. The Riemann-Hurwitz formula gives,
$$2g-2=2(2h-2)+2d\text{ i.e. }$$
$$d=g-2h+1 \text{ or equivalently }g=d+2h-1\,.$$
Also, we have an element $r\in H^0(L^2)$, unique upto non-zero constants such that $\text{div}\,r$ is the branch locus. Since $C$ is smooth, $\text{div}\,r$ is a sum of $2d$ distinct points.

\begin{lemma}\label{e=d}
 Let $p:C\ra D$ be a double cover of smooth curves as above. 
  Let $A$ be a degree $e$ line bundle on $C$
with $e \leq d$, which is basepoint free and not composed with $p$ (i.e. $A$ is not a pull back of a line bundle from $D$). Then $h^0(A) = 2$, $e = d$ and $p_*A = \mcO_D^2$.
\end{lemma}
\begin{proof}
Since $A$ is basepoint free, $e \geq 0$. If $e = 0$, then $A = \mcO_C$ and thus trivially composed with $p$. So assume $e > 0$, again the basepoint free property of $A$ gives $h^0 ( A ) \geq 2$. From Riemann-Roch, we get $\chi(A)=e-g+1=e-(d+2h-1)+1=e-d+2-2h$. Thus, $\chi( p_*A ) = e-d+2-2h=e-d+2(1-h)$, which gives $\text{deg}\,  p_* A = e-d\leq 0$. If the subsheaf generated by global sections of $p_*A$ had rank two, taking two general sections, we have an inclusion $\mcO_D^2 \hra p_*A$ and since $\text{deg}\,p_*A\leq 0$ this degree must be zero and so $e = d$ and this map is an isomorphism, thus $p_*A = \mcO_D^2$ showing that $h^0( C, A ) = h^0 ( D, p_* A ) = 2$. 
 
 If the rank of the subsheaf generated by global sections is one, taking its saturation, we have a line subbundle $M\subset p_*A$ and $H^0(M)=H^0(p_*A)$. So, one has a map $p^*M\ra A$, which can not be zero, since sections of $A$ all come from $M$. But since $A$ is globally generated $p^*M = A$, contradicting our hypothesis that $A$ is not a pull back from $D$.
\end{proof}
The proof of the following lemma is well understood.
\begin{lemma}\label{Z 3 cases}
 Let $\psi:Z\ra D$ be a finite flat map of degree 2 (where $D$ is a smooth projective curve as before). Then one of the following holds.
 \begin{enumerate}
  \item $Z$ is reduced and irreducible.
  \item $Z$ is reduced but not irreducible. Then, $Z=Z_1\cup Z_2$ where $Z_i$ irreducible and the restriction maps $Z_i\ra D$ are isomorphisms.
  \item $Z$ is not reduced. Then the map $Z_{\emph{red}} \ra D$ is an isomorphism. 
 \end{enumerate}
\end{lemma}
We now identify conditions for the existence of line bundles $A$ on $C$ as in the above lemma.
\begin{thm}\label{pencil}
 With notations as above, the following are equivalent.
 \begin{enumerate}
  \item There exists a line bundle $A$ of degree $d$ on $C$ which is basepoint free not composed with $p$.
  \item There exist elements $l,m,a\in H^0(L)$ such that the branch locus $\emph{div}\,r$ is given by $r=lm+a^2$.
 \end{enumerate}
\end{thm}
\begin{proof}
  First we show that (1) implies (2). Let $\sigma$ be the involution on $C$ such that $C/\langle \sigma\rangle = D$. We will consider the two cases separately, when $A\neq\sigma^*A$ or $A = \sigma^*A$.
  
  Assume first that $A\neq \sigma^*A$. By Lemma \ref{e=d}, we have $h^0(A)=2$. Taking two sections of $A$ and taking their conjugate sections in $\sigma^*A$, we get a morphism 
$q:C\ra \mbb{P}^1\times \mbb{P}^1$. The involution acts on $\mbb{P}^1\times \mbb{P}^1$ by switching factors and then the morphism $q:C\ra \mbb{P}^1\times\mbb{P}^1$ is
equivariant for the action of $\sigma$. We get the following commutative diagram when we take quotient by $\sigma$.
\begin{displaymath}
 \xymatrix{ C\ar[r]^{p}\ar[d]_{q} & \frac{C}{\langle\sigma\rangle}=D\ar[d]^{q'}\\
            \mbb{P}^1\times \mbb{P}^1\ar[r]_{p'} & \frac{\mbb{P}^1\times \mbb{P}^1}{\langle\sigma\rangle}=\mbb{P}^2}
\end{displaymath}
Note that 
$$\text{deg}\,q^*p'^*\mcOP(1)=\text{deg}\,q^*\mcO_{\mbb{P}^1\times\mbb{P}^1}(1,1)=\text{deg}\,A\otimes\sigma^*A=2d\,.$$
Hence, $\text{deg}\,p^*q'^*\mcOP(1)=2d$, and since $p$ is a degree two map, we have $\text{deg}\,q'^*\mcOP(1)=d$.

Let us denote $B:=q'^*\mcOP(1)$. We claim that $B=L$. From the diagram, we see that $p^*B=A\otimes\sigma^*A$. Computing global sections,
\begin{eqnarray*}\label{nocall11}
H^0(C,A\otimes\sigma^*A) & = & H^0(C,p^*B)\ {}
\nonumber\\
& = & H^0(D,B\otimes p_*\mcO_C) {}
\nonumber\\
& = & H^0(D,B)\oplus H^0(D,B\otimes L^{-1})\,.
\end{eqnarray*}
If $B\neq L$, then $B\otimes L^{-1}$ is a non-trivial degree zero line bundle and so, $H^0(D,B\otimes L^{-1})=0$. That is, all sections of $A\otimes \sigma^*A$ are invariant under $\sigma$, which can happen only if the action of $\sigma$ on $\mbb{P}^1\times\mbb{P}^1$ is trivial. Hence, we have a contradiction and $L=B$.

Consider the fibered product $Z:=D\times_{\mbb{P}^2}(\mbb{P}^1\times\mbb{P}^1)$. We then have natural morphisms $\xi:Z\ra \mbb{P}^1\times\mbb{P}^1$, $\psi:Z\ra D$ and $\phi:C\ra Z$.
\begin{displaymath}
 \xymatrix{ C\ar[ddr]_q\ar[rrd]^{p} \ar@{-->}[dr]_{\phi} & & \\
            & Z\ar[r]_{\psi}\ar[d]^{\xi} & D\ar[d]^{q'}\\
            & \mbb{P}^1\times \mbb{P}^1\ar[r]_{p'} & \mbb{P}^2}
\end{displaymath}
Note that $Z$ is reduced and irreducible. Indeed, we prove this by contradiction by using Lemma \ref{Z 3 cases}. Suppose $Z$ is as in case (2) of the Lemma, i.e. $Z=Z_1\cup Z_2$. Then, since $C$ is irreducible, $\phi:C\ra Z$ must map to one of the $Z_i$s, say $Z_1$. Let $\xi':Z_1\ra \mbb{P}^1\times\mbb{P}^1$ be 
the restriction of $\xi$ to $Z_1$. So, we have $A=\phi^*\xi'^*\mcO_{\mbb{P}^1\times\mbb{P}^1}(1,0)$. But, 
$\psi:Z_1\ra D$ is an isomorphism implies $A$ is a pull back from $D$, which we have assumed is not the case. Thus, case (2) of Lemma \ref{Z 3 cases} is not possible. If we are in case (3) of the Lemma, again, the map $\phi:C\ra Z$ factors
through $Z_{\text{red}}$ and the same argument as above shows that this case is again not possible. 

Thereby, $Z$ is reduced and irreducible. Then, $\text{deg}\,\phi =1$ and $C$ is just the normalization of $Z$. The morphism $p'$ is branched along a smooth quadric  $Q\subset\mbb{P}^2$. Restricting to the complement, we see that
$\psi$ is etale on $D\setminus {q'}^{-1}(Q)$ and this open set of $Z$ is smooth. Thereby $\phi$ is an isomorphism on $\psi^{-1}(D\setminus {q'}^{-1}(Q))$, which inturn implies that $p$ is an etale double cover outside ${q'}^{-1}(Q)$. Since $\text{deg}\,q'^*\mcOP(1) = d$, ${q'}^{-1}(Q)$ has degree $2d$. Since $p$ is branched along a divisor of degree $2d$ and is contained in ${q'}^{-1}(Q)$, we see that ${q'}^{-1}(Q)$ must be precisely the branch locus of $p$. But the branch locus
is given by $r\in H^0(L^2)$. Clearly $r = lm + a^2$ , by pulling back
 the equation of the quadric $Q$ to $D$.
 
 Next we consider the case $A =\sigma^*A$. Take a general section $s$ of A. If $s$ and $\sigma^*s$ are linearly dependent, $\text{div}\,s$ is invariant
under $\sigma$. Since $s$ was general, we may assume this divisor is a sum of distinct points and none of them is in the ramification locus. But then clearly $\text{div}\,s = p^*p ( \text{div}\, s )$ and then $A = p^* (\mcO_D (p ( \text{div}\, s )))$ is
the pull back of a line bundle from $D$, contradicting our hypothesis. Thus, $s$ and $\sigma^*s$ are linearly independent, and since $h^0(A)=2$ by Lemma \ref{e=d}, we may assume that $s$ and $\sigma^*s$ generate $H^0(A)$. This
gives a map $q:C\ra \mbb{P}^1$ which equivariant for the action of $\sigma$, where $\sigma$ acts on $\mbb{P}^1$ by switching coordinates. As before, we have a commutative diagram.
\begin{displaymath}
 \xymatrix{C \ar[r]^{p} \ar[d]_q & \frac{C}{\langle \sigma\rangle}=D\ar[d]^{q'}\\
            \mbb{P}^1\ar[r]_{p'} & \frac{\mbb{P}^1}{\langle \sigma\rangle}=\mbb{P}^1 }
\end{displaymath}
It is immediate that $\text{deg}\, q' = d$. Let $Z=D\times_{\mbb{P}^1}\mbb{P}^1$, and we use the notation as earlier. 
\begin{displaymath}
 \xymatrix{ C\ar[ddr]_q\ar[rrd]^{p} \ar@{-->}[dr]_{\phi} & & \\
            & Z\ar[r]_{\psi}\ar[d]^{\xi} & D\ar[d]^{q'}\\
            & \mbb{P}^1\ar[r]_{p'} & \mbb{P}^1}
\end{displaymath}
Exactly as in the previous case, $Z$ is reduced irreducible and $\phi:C\ra Z$ has degree 1. Since $p'$ is branched along two points, a quadric $Q$ (given by product of two linear forms) as in the previous paragraph, one checks
that the map $\psi:Z\ra D$ is etale outside $q'^{-1}(Q)$ and so by degree
consideration, we see that $q'^{-1}(Q)$ is the branch locus of $p$ and hence $r=q'^{-1}(Q)$ is a section of $L^2$. 

As before, we check that $q'^*\mcO_{\mbb{P}^1}(1)=L$.
\begin{eqnarray*}\label{nocall11}
H^0(C,A) & = & H^0(\mbb{P}^1,q^*p'^*\mcO_{\mbb{P}^1}(1))\ {}
\nonumber\\
& = & H^0(p^*q'^*\mcO_{\mbb{P}^1}(1)) {}
\nonumber\\
& = & H^0(D,q'^*\mcO_{\mbb{P}^1}(1))\oplus H^0(D,q'^*\mcO_{\mbb{P}^1}(1)\otimes L^{-1})\,.
\end{eqnarray*}
If $q'^*\mcO_{\mbb{P}^1}(1)\neq L$, then $H^0(C,A)=H^0(D,q'^*\mcO_{\mbb{P}^1}(1))$ and all sections of $A$ are $\sigma$ invariant which is not possible. Hence, $q'^*\mcO_{\mbb{P}^1}(1)= L$. Thus, $r\in H^0(L^2)$ which is obtained by pulling back $Q$ is given by an equation $r=lm$ for two sections $l,m\in H^0(L)$.

We now prove the converse (2) implies (1). So assume $r =lm+a^2$ for $l,m,a\in H^0(L)$. Since $\text{div}\,r$
is smooth, we see that $l, m, a$ generate $L$, because if they all vanished at a point, then $r$ would vanish doubly at that point. Thus, the sections $l,m,a$ give
a morphism $q:D \ra \mbb{P}^2$, with $q^*\mcOP(1)= L$. By abuse of notation, we think of $l,m,a$ as sections from $ H^0(\mbb{P}^2,\mcOP(1))$ and $r$ as a quadric in $\mbb{P}^2$. 

The dimension of the subspace generated by $l,m,a\in H^0(\mbb{P}^2,\mcOP(1))$ can not be zero and can not be one,
since they generate $L$, a positive degree line bundle. Hence $h^0(\mbb{P}^2,\mcOP(1))=2 \text{ or } 3$. If it is three,
then $lm+a^2$ is a non-singular quadric and if it is two, then we may assume (by abuse of notation) it is $r=lm$ (any homogeneous polynomial in two variables is product of linear factors). Also, since $\text{div}\, r$ on $D$ is smooth, we see
that in the latter case, the point of intersection of $l, m$ is not in the image of $D$.

We first look at the case when $l, m, a$ are linearly independent and so $r = lm + a^2$ is a smooth quadric on $\mbb{P}^2$. We have a fiber product diagram, with $p':\mbb{P}^1\times\mbb{P}^1 =Z\ra\mbb{P}^2$, the double cover branched along $r = 0$.
\begin{displaymath}
 \xymatrix{ C \ar[r]^{p}\ar[d]_{q'} & D\ar[d]^q \\
            Z \ar[r]_{p'}  & \mbb{P}^2}
\end{displaymath}
The two projections $C\ra \mbb{P}^1$ give two globally generated line bundles $A$ and $B$ on $C$. Since $A,B$ are globally generated their degrees are non-negative. In fact the degree is not zero. Indeed, if $\text{deg}\,A=0$, then the projection
corresponding to $A$ is constant and thus the map $q'$ factors through $\{t\}\times\mbb{P}^1$ for some $t\in\mbb{P}^1$. But then $p'\circ q'(C)$ is contained in a line
in $\mbb{P}^2$ . This means $q ( D )$ is contained in a line, which contradicts our assumption that $l, m, a$ are linearly independent.

So $\text{deg}\,A>0$ and $\text{deg}\,B>0$. Also,
$$A\otimes B=q'^*\mcO_{\mbb{P}^1\times\mbb{P}^1}(1,1)=q'^*p'^*\mcOP(1)=p^*q^*\mcOP(1)\,.$$
This says $\text{deg}\,A\otimes B= 2d$ and $A\otimes B$ is a pull back from $D$. Since both
$A, B$ are of positive degree, we see that at least one of them must have degree at most $d$.

We claim that neither $A$ nor $B$ is a pull back from $D$. If one of them is, then both are, since $A\otimes B$ itself is a pullback from $D$. Let $A=p^*A'$ and $B=p^*B'$. We have $\text{deg}\, A'$ , $\text{deg}\, B'$ are both less than $d$. Then
\begin{eqnarray*}\label{nocall11}
H^0(C,A) & = & H^0(D,p_*A)\ {}
\nonumber\\
& = & H^0(D,p_*p^*A') {}
\nonumber\\
& = & H^0(D,A')\oplus H^0(D,A'\otimes L^{-1})\,.
\end{eqnarray*}
Since $\text{deg}\,( A'\otimes L^{-1}) < 0$, the last term above is just $H^0 ( D, A' )$ . So, all sections of $A$ descend to $A'$ and then these
sections generate $A'$ , since they generate $A$. Applying this to $B, B'$ ,
we get a morphism $\phi:D\ra \mbb{P}^1\times\mbb{P}^1$ such that $\phi\circ p=q'$. By universal property of fiber product, we get a morphism $D\ra C$ which is absurd. Hence, neither $A$ nor $B$ is a pull back from $D$.

If $\text{deg}\,A\leq d$, applying Lemma \ref{e=d}, we see that $\text{deg}\, A = d$ which proves what we need.

Next assume $r=lm$. The proof is similar in this case as well. We have seen that $l, m$
will generate $L$ and thus we get a morphism $q:D\ra\mbb{P}^1$ with $q^*\mcO_{\mbb{P}^1}(1)=L$ and $l=0$ defines 0 and $m=0$ defines $\infty$. Take
the double cover $\mbb{P}^1\ra\mbb{P}^1$ simply ramified at zero and $\infty$. Then, the
pull back is just $C$ and we have a morphism $C\ra\mbb{P}^1$, resulting in the following Cartesian diagram.
\begin{displaymath}
 \xymatrix{ C\ar[r]^{p} \ar[d]_{q'} & D\ar[d]^q \\
            \mbb{P}^1\ar[r]_{\pi'} & \mbb{P}^1}
\end{displaymath}
Again this gives a line bundle $A=q'^*\mcO_{\mbb{P}^1}(1)$ of the desired type. We just
need to check is that $A$ is not a pull back from $D$. Clearly $A \neq \mcO_C$ and as before, if $A= p^*A'$ then $H^0(A) = H^0(A')$ . Thus, we get a morphism $\phi:D\ra \mbb{P}^1$ such that $\phi\circ p = q'$.  Then we get a morphism $D\ra C$ by universal property of fiber products, which is absurd.
\end{proof}

\subsection{Rank two Ulrich bundles on double covers of the plane}
We are now ready to give a proof of Theorem \ref{UB on the plane}.
\begin{proof}[Proof of Theorem \ref{UB on the plane}]
 Let $\pi:X\ra\mbb{P}^2$ be a smooth double cover and $\text{deg}\,B=2d$. Then by the results in $\mathcal{x}\,$ \ref{SSP section} we 
 see that the branch curve $B=(F=0)$ can be written as 
 $F=F_1G_1+F_2G_2$ where $F_i,G_i$ are homogeneous polynomials of degree $d$ and $F_1$ is smooth, and $F_2$ and $G_2$ intersect transversely at $d^2$ points.
 
 Now let $D=(F_1=0)$. Then the branch locus $B$ when restricted to $D$ is given by $r=(F_2G_2=0)$ which is in the form $r=lm+a^2$ as required by Theorem \ref{pencil}. Let $C=X\times_{\mbb{P}^2} D$. Then
  $C$ is a a double cover of $D$ branched along $\text{div}\,r$ and we have a morphism $p:C\ra D$. By Theorem \ref{pencil}, there is a basepoint free line bundle $A$ on $C$ of degree $d$ not composed with $p$. Lemma \ref{e=d} tells us that $p_*A=\mcO_D^2$.
  
 Thus by Theorem \ref{induction theorem}, $X$ admits
 a rank two Ulrich bundle.
\end{proof}
 Theorem \ref{UB on the plane} settles the questions of
 existence  and minimal rank of Ulrich bundles on all smooth double covers of $\mbb{P}^2$.
\section{Ulrich bundles on double covers of $\mbb{P}^n$ when $n\geq 3$}\label{higher double covers}
In this section we show that any smooth double cover of $\mbb{P}^n$ admits an Ulrich bundle.
The proof depends on the result of Herzog, Ulrich and Backelin \cite{HUB}  that any complete intersection variety 
admits an Ulrich sheaf. We reproduce their proof in the special case of quadric hypersurfaces. We begin with the following
\begin{lemma}\label{quadric matrix}
 Let $X$ be an irreducible quadric hypersurface defined by $(Q=0)$ in $\mbb{P}^n$ where $n\geq 3$ and
 $Q\in \mbb{C}[x_0,x_1,\ldots,x_n]$. If there is a $2^s\times 2^s$ matrix $A$ on $\mbb{P}^n$
  with linear entries such that $A^2=Q\cdot \emph{Id}$, 
  then $X$ admits an Ulrich sheaf of rank $2^{s-1}$.
  
\end{lemma}
\begin{proof}
 Let $F=\mcPn^{2^s}$. Then we have the following short exact sequence
 on $\mbb{P}^n$
 \begin{equation}\label{basicex}
   0\ra F(-1)\xra{A} F\ra G\ra 0\,.
 \end{equation}
 Note that multiplication by $A$ is injective, since multiplication by $A^2$ is. Then consider the following diagram:
\begin{displaymath}
 \xymatrix{
  &  &  & \text{`Kernel'}\ar[d]  & \\
0\ar[r] & F(-2)\ar[r]^{A^2}\ar[d] & F \ar[r]\ar@{=}[d] & F|_X \ar[d]\ar[r] & 0 \\
0 \ar[r] & F(-1)\ar[r]^{A} \ar[d] & F\ar[r] & G\ar[r] & 0 \\
& G(-1) &  & & }
\end{displaymath}
By snake lemma, $\text{`Kernel'} \simeq G(-1)$. Hence we have
the following short exact sequence on $X$
$$0\ra G(-1)\ra F|_X\ra G\ra 0\,.$$
Since $F$ has rank $2^s$ on $X$, the sheaf $G$ on $X$ has rank
$2^{s-1}$. We now show that $G$ is Ulrich on $X$. By projection, we
have a degree two morphism
$$\pi:X\ra \mbb{P}^{n-1}$$
So, on $\mbb{P}^{n-1}$ we have a short exact sequence:
\begin{equation}\label{2exseq}
  0 \ra \pi_* G(-1) \ra \pi_* \mcO_X^{2^s} \ra \pi_* G \ra 0\ .
\end{equation}
Notice that since $X$ is a quadric,
$$\pi_* \mcO_X = \mcO_{\mbbP^{n-1}} \oplus \mcO_{\mbbP^{n-1}}(-1)\,.$$
Next, using the exact sequence \eqref{basicex}, we see that
$H^i(G(\star)) = 0$ for $0 < i < \text{dim} X = n-1$. This says $H^i(\pi_* G(\star)) = 0$ for $0 < i < n-1$, and so
by Horrocks' criterion (cf. \cite{Hor} or \cite[Theorem 2.3.1]{OK}), $\pi_* G$ is a direct sum of line
bundles. Further, since $n - 1 \geq 2$, we have $\text{Ext}^1(\pi_*G,\pi_*G(-1))=0$, 
and thereby the exact sequence \eqref{2exseq}
splits:
\begin{equation}
  \pi_*G \oplus \pi_* G(-1) \simeq \mcO^{2^s}_{\mbbP^{n-1}} \oplus \mcO_{\mbbP^{n-1}}(-1)^{2^s}\,.\nonumber
\end{equation}
Hence $\pi_*G = \mcO_{\mbbP^{n-1}}^\alpha \oplus \mcO_{\mbbP^{n-1}}(-1)^\beta$
for some $\alpha$ and $\beta$. 

But, by projection formula,
$$\pi_*G(-1)=\pi_*(G\otimes\mcPn(-1)|_X)\simeq \pi_*(G\otimes\pi^*\mcO_{\mbb{P}^{n-1}}(-1))\simeq \pi_*G\otimes \mcO_{\mbb{P}^{n-1}}(-1)\,.$$
Thereby the exponent $\beta$ must be zero in the expression of $\pi_* G$  since
otherwise $\pi_* G(-1)$ will have $\mcO_{\mbbP^{n-1}}(-2)$ as a summand
and cannot be a
direct summand of $\mcO^{2^s}_{\mbbP^{n-1}} \oplus \mcO_{\mbbP^{n-1}}(-1)^{2^s}$.

So, $\pi_*G = \mcO^{2^s}_{\mbbP^{n-1}}$ showing $G$ is Ulrich of rank
$2^{s-1}$.
\end{proof}
\begin{thm}\label{quadric hypersurface}
  Let $X$ be an irreducible quadric hypersurface defined by $(Q=0)$ in $\mbb{P}^n$
  where $n\geq 3$ and $Q\in \mbb{C}[x_0,x_1,\ldots,x_n]$. Then $X$
  admits an Ulrich sheaf.
\end{thm}
\begin{proof}
  Note that there exists a linear subspace $\mbb{P}^{n_0}\subset\mbb{P}^n$ and a 
  smooth quadric in $\mbb{P}^{n_0}$ such that $X$ is a cone over this smooth
  quadric. In particular, $Q$ involves only variables from $\mbb{P}^{n_0}$ and 
  involves all of them. Depending on the parity of $n_0$, we can write
  $$Q=l_1m_1+l_2m_2+\cdots+ l_sm_s\,,$$
  where $l_i$'s form a basis of $H^0(\mbb{P}^{n_0},\mcO_{\mbb{P}^{n_0}}(1))$ and
  \begin{enumerate}
  \item  if $n_0=2p$, then $l_s=m_s$ and $s=p+1$;
  \item  if $n_0=2p+1$, then $s=p+1$.
  \end{enumerate}
  In either case, $s=\big[\frac{n_0}{2}\big]+1$.
  
  We claim that there is an Ulrich sheaf of rank $2^{s-1}$ on $X$. By
  Lemma \ref{quadric matrix}, it is enough to find a matrix
  $A$ of size $2^s$ on $\mbbP^n$ with linear entries such that
  $A^2 = Q \cdot \text{Id}$. We construct such an $A$ by induction 
  on $s$. If $Q = l_1 m_1$, take
  \begin{equation}
    A = \begin{bmatrix} 0 & l_1 \\ m_1 & 0 \end{bmatrix}\,.\nonumber
  \end{equation}
  Suppose that we have constructed a matrix $A_1$ for
  $Q_1 = l_1 m_1 + \cdots + l_p m_p$ with $A_1^2 = Q_1 \cdot \text{Id}$ and 
  if $Q = l_1 m_1 + \cdots + l_p m_p + l_{p+1} m_{p+1}$, then, take
  \begin{equation}
    A = \begin{bmatrix} A_1 & l_{p+1} \text{Id} \\ m_{p+1} \text{Id} & -A_1 \end{bmatrix}\, ,
  \end{equation}
  which indeed satisfies $A^2 = Q\cdot \text{Id}$.
\end{proof}
We can now prove that any smooth double cover of the projective space admits an Ulrich bundle.
\begin{proof}[Proof (I) of Theorem \ref{Ulrich bundle on higher Pn}]
Let $N:={d+n\choose n}-1$ and consider the $d$-tuple embedding 
 $$i:\mbb{P}^n\hra \mbb{P}^N\,.$$
 Under this embedding $i^*\mcO_{\mbbP^N}(1)\simeq \mcalP(d)$ and the morphism 
 $$H^0(\mbb{P}^N,\mcO_{\mbbP^N}(2))\ra H^0(\mbb{P}^n,\mcalP(2d))$$ 
 is surjective since the embedding is projectively normal.
 Thus there exists a quadric hypersurface $\tildeB\subset\mbb{P}^N$ such that
 the degree $2d$ hypersurface $B\subset \mbb{P}^n$ is the restriction of $\tildeB$, 
 i.e. $B=\tildeB \cap \mbb{P}^n$. The hypersurface $\tildeB$ 
 may not be smooth. Consider the double cover $\tildeX$ of $\mbb{P}^N$ branched along $\tildeB$.
 Then we have the following commutative diagram.
 \begin{displaymath}
  \xymatrix{X \ar@{^{(}->}[r]^i\ar[d]_{\pi} & \tildeX\ar[d]^{\pi'} \\
 \mbb{P}^n \ar@{^{(}->}[r]_{j} & \mbb{P}^N}
 \end{displaymath}
Note that $\tildeX$ is a quadric hypersurface in $\mbb{P}^{N+1}$. In fact if the hypersurface $\tildeB\subset \mbb{P}^N$ is defined by the equation $(F=0)$ where $F\in \mbb{C}[Y_0,Y_1,\ldots, Y_N]$ is a quadric, then $\tildeX$ is given by the equation $(T^2-F=0)$ where the homogeneous coordinate ring of $\mbb{P}^{N+1}$ is $\mbb{C}[Y_0,Y_1,\ldots, Y_N][T]$.
By Theorem \ref{quadric hypersurface}, $\tildeX$ admits an Ulrich sheaf, say $E'$. Hence
$\pi'_*E'$ is a trivial vector bundle on $\mbb{P}^N$.

Consider the pullback $i^*E'$ to $X$.
This is a sheaf satisfying $\pi_*i^*E'=j^*\pi'_*E'$ \cite[\href{https://stacks.math.columbia.edu/tag/02KE}{Tag 02KE}]{SP}. Since the direct image of $E'|_X$ 
under $\pi$ is trivial, $E'|_X$ is Maximal Cohen Macaulay on $X$. But an MCM module over a regular
local ring is free, and thereby $E'|_X$ is locally free,  and an
Ulrich bundle on $X$.
\end{proof}
We can also use Theorem \ref{induction theorem} to prove that every smooth double cover of $\mbb{P}^n$ admits an Ulrich bundle. We first need a Lemma, whose proof is very similar to the above Proof (I) of Theorem \ref{Ulrich bundle on higher Pn}. 
\begin{lemma}\label{vb direct image}
 Let $f:C\ra D$ be a double cover of smooth projective curves such that the branch locus $B\in |L^2|$ where
 $L$ is a very ample line bundle on $D$.
Assume that the multiplication map 
\begin{equation}\label{proj-eq-step1}
H^0(L)\otimes H^0(L)\ra H^0(L^2) 
\end{equation}
is surjective. 
  Then there is a vector bundle $E$ on $C$ such that $f_*E$ is trivial. 
\end{lemma}
\begin{proof}
 Let $V=H^0(L)^* $. The very ample line bundle $L$ embeds $D$ in $\mbb{P}(V)$. Under this embedding
 $i:D \hra \mbb{P}(V)$,  the branch locus $B$
 is the restriction of a quadric $\tildeB\subset\mbb{P}(V)$ to $D$ by the surjection \eqref{proj-eq-step1}. 
 Consider the double cover $X$ of $\mbb{P}(V)$
 branched along $\tildeB$, then the curve $C$ is simply the fibred product $C=X\times_{\mbb{P}(V)} D$
 and we have
 \begin{displaymath}
  \xymatrix { C\ar@{^{(}->}[r]\ar[d]_{f} & X\ar[d]^{f'} \\
            D \ar@{^{(}->}[r]_i & \mbb{P}(V)}
 \end{displaymath}
Since $f':X\ra\mbb{P}(V)$ is a double cover branched over a quadric,
$X$ is a quadric hypersurface in $\mbb{P}(V\oplus\mbb{C})$. 
The rest of the proof is similar to Proof (I) of Theorem \ref{Ulrich bundle on higher Pn}.
\end{proof}
This lemma leads us to an alternate way of constructing Ulrich
bundles on double covers of projective spaces.
\begin{proof}[Proof (II) of Theorem \ref{Ulrich bundle on higher Pn}]
Choose sections $s_1,s_2,\ldots, s_{n-1}\in H^0(\mbb{P}^n,\mcPn(d))$ such that $Y_i:=Z(s_1)\cap Z(s_2)\cap\cdots\cap Z(s_i)$ is a general irreducible smooth complete intersection inside $\mbb{P}^n$. Here by general we that, if $X_i$ denotes the fibred product $X_i=X\times_{\mbb{P}^n} Y_i$, then $\pi_i:X_i\ra Y_i$ is a double cover. In particular, we have the following commutative diagram.
\begin{displaymath}
 \xymatrix{X_{n-1}\ar[d]^{\pi_{n-1}}\ar@{^{(}->}[r]\ar[d] & X_{n-2}\ar[d]^{\pi_{n-2}}\ar@{..}[r] &  X_1 \ar@{^{(}->}[r] \ar[d] & X\ar[d]^{\pi}\\
 Y_{n-1} \ar@{^{(}->}[r] & Y_{n-2}\ar@{..}[r]  &  Y_1 \ar@{^{(}->}[r] & \mbb{P}^n}
\end{displaymath}
 So $Y_{n-1}$ is a smooth complete intersection curve and $\pi_{n-1}$ is a double cover
of smooth curves branched along a divisor, say $B_{n-1}\in |\mcPn(2d)|_{Y_{n-1}}|$. The line bundle $L=\mcPn(d)|_{Y_{n-1}}$ is very ample and satisfies the condition that
$$ H^0(L)\otimes H^0(L)\ra H^0(L^2)\text{ is surjective}\,.$$
Hence, there is a vector bundle $V$ on $X_{n-1}$ such that 
${\pi_{n-1}}_*V$ is trivial. 
Further, for $1\leq i\leq n-2 $, by
induction $Y_i$ also satisfies 
$H^1(Y_i,\mcO_{Y_i})=0$. 
Thereby, we can apply Theorem \ref{induction theorem} and construct a vector bundle $E$ on $X$ whose direct image is trivial. 
\end{proof}
Consider a smooth double cover $\pi:X\ra \mbb{P}^n$ branched along a smooth
hypersurface $B\in |\mcPn(2d)|$. Then $B$ is defined by a degree $2d$ homogeneous
polynomial $F\in \mbb{C}[x_0,x_1,\ldots,x_n]$.
\newtheorem{prop}[thm]{Proposition}
\begin{prop}\label{prop n-1 complete int}
 Assume that $F$ can be written in the form 
 $$F=F_1G_1+F_2G_2+\cdots+ F_nG_n\,,$$
 for $F_i,G_i\in |\mcPn(d)|$,  $i=1,2,\ldots,n$, where
 \begin{itemize}
  \item $Y_i=\cap_{j=1}^i Z(F_j) \text{ for } i=1,2,\ldots,n-1$ 
define general smooth irreducible complete intersection varieties of type $(d,d,\ldots,d)$.
\item $(F_nG_n=0)$ defines a smooth subvariety of $Y_{n-1}$.
 \end{itemize}
 Then $X$ admits an Ulrich bundle of rank $2^{n-1}$.
\end{prop}
\begin{proof} 
Consider the following fibred diagram
 \begin{displaymath}
 \xymatrix{X_{n-1}\ar@{^{(}->}[r]\ar[d]^{p^{n-1}} & X_{n-2}\ar[d]^{p^{n-2}}\ar@{..}[r] & X_1 \ar@{^{(}->}[r] \ar[d]^{p^1} & X\ar[d]^{\pi}\\
 Y_{n-1} \ar@{^{(}->}[r] & Y_{n-2}\ar@{..}[r]  &  Y_1 \ar@{^{(}->}[r] & \mbb{P}^n}
\end{displaymath}
Without loss of generality, one may assume that, $p^i:X_i\ra Y_i$ is a double cover. 
 Note that 
each $p^i:X_i\ra Y_i$ is branched along $B\cap Y_i\in |\mcPn(2d)|_{Y_i}|$,

   In particular, the branch locus $B\cap Y_{n-1}$ of
$p^{n-1}:X_{n-1}\ra Y_{n-1}$ is given by $(F_nG_n=0)$ since $(F_i=0)$ for $i=1,2,\ldots, n-1$ on 
$Y_{n-1}$. Thereby $p^{n-1}:X_{n-1}\ra Y_{n-1}$ is a double cover of smooth curves whose branch locus has the form $r=lm+a^2$. By Theorem \ref{pencil}, there is a line bundle $A$ on $X_{n-1}$ whose direct image is trivial. 
Since $H^1(Y_i,\mcO_{Y_i})=0$ for $1\leq i\leq n-2 $, the repeated application of Theorem \ref{induction theorem},
gives an Ulrich vector bundle $E$ of rank $2^{n-1}$ on $X$.
\end{proof}

Both proofs of Theorem \ref{Ulrich bundle on higher Pn} in this section do not make any assertions on the rank of the Ulrich bundle on double covers of $\mbb{P}^n$ so obtained. We deal with that in the next section.
\section{Estimating the rank of Ulrich bundles on double covers of higher $\mbb{P}^n$}\label{rank section}

 From the previous sections, we see that the rank of the Ulrich bundle obtained on the double cover depends on the expression of the degree $2d$ branch locus $F$ in the form
$$F=\Sigma_i F_iG_i$$
where $F_i$ and $G_i$ are degree $d$ homogeneous forms. 
  We now use this to obtain an upper bound for the rank of an Ulrich bundle $E$ on a double cover 
$\pi:X\ra\mbb{P}^n$ that we constructed in the Proof (I) of Theorem \ref{Ulrich bundle on higher Pn}.

Consider the space $V:=H^0(\mbb{P}^n,\mcPn(2d))$ of degree $2d$ polynomials
in $n+1$ variables $x_0,x_1,\ldots, x_n$, and its projectivization $\mbb{P}V$.
We saw in $\mathcal{x}\,$\ref{Secant varieties prelim} that there is a stratification of 
$\mbb{P}V$ by the secant varieties $\text{Sec}_r\mbb{X}$, where 
$$\mbb{X}=\{[F]\in \mbb{P}V\,|\,F=F_1F_2, F_i\in H^0(\mbb{P}^n,\mcPn(d))\}\,.$$
In particular, 
we have a chain of inclusions
$$\mbb{X}\subset \text{Sec}_{1}\mbb{X}\subset\cdots\subset 
\text{Sec}_{r-1}\mbb{X}\subset \text{Sec}_{r}\mbb{X}\subset\cdots\subset\mbb{P}V\,$$ and
if $[F]\in \text{Sec}_{r-1}\mbb{X}$, then $F$ can be
written as
$F=\Sigma_{i=1}^r F_i G_i\ \text{where }F_i,G_i\in H^0(\mbb{P}^n,\mcPn(d))\,.$

Note that any degree $2d$ hypersurface in $\mbb{P}^n$ can be thought of as a pullback of a quadric from $\mbb{P}^N$, where $\mbb{P}^n\hra\mbb{P}^N$ is the $d$-uple embedding. Then, the expression of a quadric as in Theorem \ref{quadric hypersurface} gives that the above filtration of $\mbb{P}V$ by secant varieties is finite, as stated in the lemma below.
\begin{lemma}\label{filtration covers}
There is a positive integer $m$ such that $\emph{Sec}_m\mbb{X}$ will cover $\mbb{P}V$. 
In particular, $m=\big[\frac{N}{2}\big]$ where $N={d+n\choose n}-1$ and the above chain of inclusions becomes
$$\mbb{X}\subset \emph{Sec}_{1}\mbb{X}\subset\cdots\subset \emph{Sec}_{\big[\frac{N}{2}\big]}\mbb{X}=\mbb{P}V\,.$$
\end{lemma}
As before, let $\pi:X\ra \mbb{P}^n$ be a smooth double cover branched along
a smooth hypersurface $B=(F=0)$ of degree $2d$. 
\begin{lemma}\label{max rank}
Let $r=\emph{min}\{s\,|\,[F]\in\emph{Sec}_s\mbb{X}\}$.
Then $X$ admits an Ulrich bundle of rank $2^{r}$ or $2^{r+1}$.
\end{lemma}
\begin{proof}
Since $r$ is the minimum integer such that $F\in \text{Sec}_r\mbb{X}$, there is an irredundant and minimal expression of $F$ in the form 
 \begin{equation}\label{Combination}
 F=\Sigma_{i=1}^{r+1} F_i G_i\,,
\end{equation}
where $F_i,G_i\in H^0(\mbb{P}^n,\mcPn(d))$.
 Let $N:={d+n\choose n}-1$ and consider the $d$-tuple embedding 
 $$i:\mbb{P}^n\hra \mbb{P}^N\,.$$
 There is a quadric hypersurface $\tildeB\subset\mbb{P}^N$ such that
 the degree $2d$ hypersurface $B\subset \mbb{P}^n$ is the restriction of $\tildeB$, 
 i.e. $B=\tildeB \cap \mbb{P}^n$. If $\tildeB=(Q=0)$, where $Q$ is a degree two polynomial in the homogeneous coordinate ring of $\mbb{P}^N$, then $Q$ has the expression:
 $$Q=\Sigma_{i=1}^{r+1} l_i m_i\,,$$
 where $l_i,m_i\in H^0(\mbb{P}^N,\mathcal{O}_{\mathbb{P}^N}(1))$ and we have the restrictions $(l_i=0)|_{\mbb{P}^n}=(F_i=0)$ and $(m_i=0)|_{\mbb{P}^n}=(G_i=0)$. This expression is irredundant and minimal, since the expression of $F$ is.
 Therefore, in the expression of $Q$ as above, we either have
 \begin{enumerate}
  \item[(a)] $l_i\neq m_i$ for all $i=1,2,\ldots,r+1$, or
  \item[(b)] $l_i\neq m_i$ for $i=1,2,\ldots,r$ and $l_{r+1}=m_{r+1}$. 
 \end{enumerate}
 The double cover $\tildeX$ of $\mbb{P}^N$ branched along $\tildeB$ is a quadric hypersurface in $\mbb{P}^{N+1}$ and is given by $(T^2-\Sigma_{i=1}^{r+1} l_i m_i=0)$. We now get the following two cases based on the Proof (I) of Theorem \ref{Ulrich bundle on higher Pn}.
 \begin{itemize}
 \item[(a)] Suppose that $Q$ has the form $Q=\Sigma_{i=1}^{r+1} l_i m_i$ where $l_i\neq m_i$ for all $i$.
  Then, the equation of $\tildeX$ is $(T^2-Q=0)$. Note that the expression of the polynomial $T^2-Q=T^2-\Sigma_{i=1}^{r+1} l_i m_i$ is itself irredundant and minimal (as a sum of products of two linear forms). Hence, the double cover $\tildeX$ and hence $X$ admit an Ulrich bundle of rank $2^{r+1}$.  
  \item[(b)] On the other hand, assume that $Q$ has the form $Q=\Sigma_{i=1}^{r} l_i m_i+ l_{r+1}^2$.
  Then the equation of $\tildeX$ which is $(T^2-Q=0)$ can be rewritten as 
  $$(T+l_{r+1})(T-l_{r+1}) -\Sigma_{i=1}^{r} l_i m_i=l_0m_0-\Sigma_{i=1}^{r} l_i m_i=0\,.$$
  Thus, in this case $\tildeX$ and $X$ admit an Ulrich bundle of rank $2^{r}$.
 \end{itemize}
\end{proof}
\begin{rmk}
 In $\S$ \ref{SSP section}, we discussed the expression of any smooth plane degree $2d$ curve $F$ in the form $F=F_1G_1+F_2G_2$. So any $[F]\in \emph{Sec}_1\mbb{X}$. Then from the proof of the above Lemma, we see that
 $r=\emph{min}\{s\,|\,[F]\in\emph{Sec}_s\mbb{X}\}=1$ in our case. Further, since the expression of a general $F$ is of the form as described in part (a) of the proof above, $X$ admits an Ulrich bundle of rank $2^{r+1}=4$. Hence, the above method does not give the minimal rank Ulrich bundle, which we obtained from the first technique cf. Theorem \ref{UB on the plane}. 
\end{rmk}
A direct application of Lemma \ref{filtration covers} gives us the following
\begin{thm}\label{max rank UB}
 Let $\pi:X\ra \mbb{P}^n$ be a smooth double cover branched along
a smooth hypersurface $B$ of degree $2d$. Then $X$ admits an Ulrich bundle
of rank $\leq 2^{\big[\frac{N}{2}\big]+1}$ where $N= {d+n\choose n}-1$. 
\end{thm}
There is a lower bound on the rank of Ulrich bundles on $\pi:X\ra \mbb{P}^n$ which 
 can be constructed using the above techniques. This is shown in the following 
 lemma and corollary.
 \begin{lemma}\label{lower bound}
  Let $F\in \mbb{C}[x_0,x_1,\ldots,x_n]$ be a homogeneous polynomial of degree $2d$ which defines a smooth hypersurface in $\mbb{P}^n$. If $F$ has an expression of the form 
  $$F=\Sigma_{i=1}^r F_iG_i\,$$
where $F_i,G_i\in \mbb{C}[x_0,x_1,\ldots,x_n]$ are homogeneous of degree $d$, then $2r\geq n+1$.
 \end{lemma}
\begin{proof}
Consider the ideal $$J=(F_1,G_1,F_2,G_2,\ldots,F_r, G_r)\subset \mbb{C}[x_0,x_1,\ldots,x_n]\,.$$
Since $F=\Sigma_{i=1}^r F_iG_i$ defines a smooth hypersurface in $\mbb{P}^n$, the zero set of $J$ in $\mbb{C}^{n+1}$ is
$$Z(J)=(0,0,..0)\,.$$
Denote $\mathfrak{m}=(x_0,x_1,\ldots,x_n)$, 
the irrelevant maximal ideal. Then by Hilbert's Nullstellensatz, 
$$\sqrt{J}=I(Z(J))=\mathfrak{m}\,.$$
This shows that $J$ is $\mathfrak{m}$-primary and consequently, $\mathfrak{m}$ is the minimal prime over $J$. Then, by Krull's principal ideal theorem \cite[Theorem 10.2]{DE}, 
$$\text{codim}\,\mathfrak{m}=n+1\leq 2r\,.$$
\end{proof}
As a direct consequence of this Lemma and Lemma \ref{max rank}, we obtain the following
\begin{coro}\label{min rank}
 Let $\pi:X\ra \mbb{P}^n$ be a double cover branched along a smooth hypersurface
  $B=(F=0)$ of degree $2d$. Then the Ulrich bundle which can be produced by the above methods on $X$ has rank $\geq 2^{r-1}$ where $2r\geq n+1$.
\end{coro}
\section{Appendix}\label{appendix}
While proving that any smooth double cover of $\mbb{P}^2$ carries a rank two Ulrich bundle, one crucial step we undertook was the analysis of double covers $p:C\ra D$ of smooth curves in order to identify when there is a line bundle $A$ on the curve $C$ whose direct image on $D$ is trivial. 

Our initial approach was to use the following theorem from the paper \cite{KO} which seemed appropriate for our situation.
\begin{thm}\cite[Theorem A]{KO}\label{Keem statement}
Let $C$ be a curve of genus $g$ which admits a double covering
$p:C\ra D$ such that the genus of $D$ is $ h\geq 0$ and $g\geq 4h$. Then $C$ has a basepoint free and complete $g^1_{g- 2 h+1}$ not composed with $p$.
\end{thm}
We applied the above Theorem in our context by considering a general degree $d$ curve $D\hra \mbb{P}^2$ and its double cover as follows.
\begin{displaymath}
 \xymatrix{C \ar@{^{(}->}[r]^i\ar[d]_{\pd} & X\ar[d]^{\pi} \\
 D \ar@{^{(}->}[r]_{j} & \mbb{P}^2}
\end{displaymath}
In this case, one can check that $\text{genus}(C)\geq 4\,\text{genus}(D)$. Theorem \ref{Keem statement} then indicates that there is a complete and basepoint free $g^1_{g- 2 h+1}$ not composed with $p$ on $C$. Such a line bundle $A$ has the property that $p_*A$ is a trivial vector bundle on $D$ , cf. Lemma \ref{e=d}. This enabled us to obtain results about the existence of Ulrich bundles on $X$.

However, as mentioned in the main text, we later identified that the proof of the Theorem \ref{Keem statement} is not quite correct. In this appendix, we provide a counterexample to Theorem \ref{Keem statement}. 
We first recall the statement of the Theorem \ref{pencil} which we have proved in $\mathcal{x}$ \ref{Ulrich bundles over surfaces}. 
\begin{thmfour}
 Let $p:C\ra D$ be a double cover of smooth curves and $p_*\mcO_C=\mcO_D\oplus L^{-1}$ for a line bundle $L$ on $D$ with $\text{degree}\,L=d=g-2h+1$ where $g$ and $h$ are the genera of $C$ and $D$ respectively. Let $r\in H^0(L^2)$, unique upto non-zero constants be such that $\emph{div}\,r$ is the branch locus. 
 \begin{enumerate}
  \item There exists a line bundle $A$ of degree $d$ on $C$ which is basepoint free not composed with $p$.
  \item There exist elements $l,m,a\in H^0(L)$ such that the branch locus $\emph{div}\,r$ is given by $r=lm+a^2$.
 \end{enumerate}
\end{thmfour}
We now describe the counterexample to Theorem \ref{Keem statement} by exhibiting a double cover $p:C\ra D$ with $\text{genus}\,(C)\geq 4\,\text{genus}\,(D)$ whose base locus $\text{div}\,r$ cannot be written in the form $r=lm+a^2$.
\subsection*{The Counterexample}
Let $D$ be a smooth projective curve of genus 2. Consider the line bundle $L = \mcO_D(2K_D + P)$ for any point $P\in D$. Let $s, t \in H^0(K_D)$ be sections which form a basis for $H^0(K_D)$.
We can check that $h^0(2K_D)=3$ and so $s^2,t^2,st\in H^0(2K_D)$ form a basis. We have a natural inclusion $H^0(2K_D)\subset H^0(L)$, so we can consider $s^2,t^2,st\in H^0(L)$ and they are linearly independent. Now $L$ is a line bundle of degree 5 and $h^0(L)=4$. Embed $D$ inside $\mbb{P}^3$ using $H^0(L)$. Since $s^2t^2=(st)^2$, the curve $D$ is contained in a quadric $Q$ of rank 3. If the homogeneous coordinates on $\mbb{P}^3$ are $x,y,z,t$, then the quadric is $xy-z^2=0$ which by change of coordinates can be written as $x^2+y^2+z^2=0$. This quadric is singular and irreducible. 

Let $V = H^0( L )$, the vector space generated by $x, y, z, t$. It is known that 
$$H^0(L)\otimes H^0(L)\ra H^0(L^2)$$ is onto \cite[Page 55, Corollary]{Mum}. Thereby we get that 
\begin{equation}\label{symmetric}
 S^2V\ra H^0(L^2)
\end{equation}
is onto. Note that $\text{dim}\,S^2V=10$ since $\text{dim}\,V=4$ and $h^0(L^2)=9$. So the kernel of \eqref{symmetric} is one dimensional. Since $D\subset Q=(x^2+y^2+z^2=0)$, the kernel is generated by $Q$. 

Suppose any element of $r\in H^0(L^2)$ can be written in the form $lm+a^2$ with $l,m,a\in V$. Then starting with an element $r\in S^2V$, there exist $l,m,a\in V$ such that $r-lm-a^2$ goes to zero in $H^0(L^2)$, i.e. $r-lm-a^2=\alpha Q$ for some $\alpha\in \mbb{C}$. That is,
$$r-\alpha Q=lm+a^2.$$
But any element of the form $lm+a^2$ is a singular
quadric. For a general $r$, we have $r-\alpha Q$ is non-singular. For instance, if $r=xy+ity+zt$, we can check that for all $\alpha\in\mbb{C}$, $r-\alpha Q$ is non-singular. 

This shows that, a general $r\in H^0(L^2)$  cannot be written in the form $lm+a^2$. Let $p:C\ra D$ be the double cover of $D$ branched along such an $r$. If $g=\text{genus}(C)$ and $h=\text{genus}(D)$, the Riemann-Hurwitz gives
$$2g-2=2(2h-2)+2d\,.$$
But $h=2$ is our assumption and $d=\text{deg}\,L=5$. Then,
$$2g-2=4+10=14\,,\text{ that is }$$
$$g=8\geq 4h.$$
So we have a double cover of smooth curves $p:C\ra D$ with $g\geq 4h$ such that the branch locus $r$ cannot be written in the form $r=lm+a^2$. This is equivalent to saying that there is no line bundle $A$ on $C$ whose direct image is trivial. This gives a counterexample to the statement of Theorem \ref{Keem statement}.



\end{document}